\documentclass[11pt,reqno]{amsart}
\usepackage{amsmath,amssymb,geometry,color,tikz-cd}
\usepackage[colorlinks=true, linkcolor=red!80!black, urlcolor=purple, citecolor=blue!70!black]{hyperref}
\usepackage{amssymb}% see geometry.pdf on how to lay out the page. There's lots.
\usepackage{bbm}
\usepackage{enumitem}
\usepackage{upgreek}
\usepackage{bm}
\usepackage{stmaryrd}
\usepackage{mathrsfs}
\allowdisplaybreaks

\newcommand{\cO}{\mathcal{O}}

\newcommand{\bP}{\mathbb{P}}
\newcommand{\bC}{\mathbb{C}}
\newcommand{\bR}{\mathbb{R}}

\newcommand{\bQ}{\mathbb{Q}}
\newcommand{\bZ}{\mathbb{Z}}

\newcommand{\bG}{\mathbb{G}}

\newcommand{\Bl}{\mathrm{Bl}}

\DeclareMathOperator{\Aut}{Aut}

\DeclareMathOperator{\Fut}{Fut}

\DeclareMathOperator{\vol}{vol}

\DeclareMathOperator{\ord}{ord}

\DeclareMathOperator{\Supp}{Supp}

\newcommand{\PGL}{\mathrm{PGL}}

\newcommand{\Hom}{\mathrm{Hom}}
\newcommand{\bT}{\mathbb{T}}

\newcommand{\oS}{\overline{S}}
\newcommand{\oH}{\overline{H}}
\newcommand{\oT}{\overline{T}}

\numberwithin{equation}{section}

\newtheorem{prop} {Proposition} [section]
\newtheorem{thm}[prop] {Theorem} 

\newtheorem{lem}[prop] {Lemma}
\newtheorem{cor}[prop]{Corollary}
\newtheorem{prop-def}[prop]{Proposition-Definition}

\newtheorem{thm-defn}[prop]{Theorem-Definition}

\theoremstyle{definition}
 
\newtheorem{rem}[prop] {Remark}

\title{K-stability of Fano threefolds of rank $2$ and degree $14$ as double covers}

\author{Yuchen Liu}
\address{Department of Mathematics, Northwestern University, Evanston, IL 60208, USA.}
\email{yuchenl@northwestern.edu}

\date{\today} % delete this line to display the current date

\begin{document}
\begin{abstract}
We prove that every smooth Fano threefold from the family \textnumero2.8 is K-stable. Such a Fano threefold is a double cover of the blow-up of $\bP^3$ at one point branched along an anti-canonical divisor.
\end{abstract}

\maketitle
%\setcounter{tocdepth}{1}
%\tableofcontents

\section{Introduction}
Every smooth Fano threefold belongs to one of the $105$ families according to the Iskovskikh-Mori-Mukai classification, see \cite{IP99, MM03} for the complete list and labeling of the families. The celebrated Calabi problem arised from differential geometry asks to find K\"ahler-Einstein metrics on Fano manifolds. From the solutions to the Yau-Tian-Donaldson Conjecture \cite{CDS15, Tia15}, the Calabi problem reduces to checking the algebraic condition, namely K-polystability, for Fano manifolds. Recently there has been much progress on the study of K-stability especially from an algebraic point of view, see \cite{Xu21} for a survey on this topic. Notably, the Calabi problem for a general Fano threefold in each of the 105 families has been solved by Araujo et al.\ in \cite{ACC+}. Nevertheless, there are still many families of Fano threefolds where the Calabi problem remains open for all smooth members. 

In this short note, we prove that every smooth member of the family \textnumero2.8 is K-stable. From now on, let $X$ be a smooth Fano threefold from the family \textnumero2.8. Then $X$ has Picard rank $2$ and degree $14$.
Let $\pi:Y=\Bl_p \bP^3\to \bP^3$ be the blow-up of $\bP^3$ at a point $p$ where $E\subset Y$ is the exceptional divisor of $\pi$. Then $X$ is a double cover $\sigma: X\to Y$ branched along a smooth anti-canonical divisor $\oS\subset Y$.  According to \cite[Section 5.1]{ACC+}, a general member of the family \textnumero2.8 is K-stable.

\begin{thm}\label{thm:main}
Every smooth Fano threefold $X$ from the family \textnumero2.8 is K-stable hence admits K\"ahler-Einstein metrics.
\end{thm}

Since $\Aut(X)$ is finite by \cite{CPS19}, it suffices to show that $X$ is K-polystable.
From \cite{Der16, LZ22, Zhu20} on K-stability of cyclic covers, we know that to show $X$ is K-polystable it suffices  to show that $(Y, \frac{1}{2}\oS)$ is K-polystable. By \cite[Theorem 2.10]{ADL21} (see also \cite[Corollary 1]{JMR16}) we know that $(Y, (1-\epsilon)\oS)$ is K-stable for $0<\epsilon\ll 1$. Thus using interpolation of K-stability \cite[Proposition 2.13]{ADL19} it suffices to show that $(Y, c\oS)$ is K-semistable for some $c\in (0, \frac{1}{2})$. This is done by finding a suitable special degeneration $(Y, c\oS)\rightsquigarrow (Y,c\oS_0)$, showing (equivariant) K-semistability of $(Y,c\oS_0)$, and then using openness of K-semistability \cite{BLX19, Xu19}.

The smooth members of the family \textnumero2.8 split into two subfamilies: \textnumero2.8(a) where $\oS\cap E$ is a smooth conic curve, and \textnumero2.8(b) where $\oS\cap E$ is a union of two transversal lines. In fact, if $\oS\cap E$ is a double line then $\oS$ cannot be smooth. We shall split the proof into two cases accordingly. 

\begin{rem}
There are two families of smooth Fano threefolds of Picard rank $2$ and degree $14$: \textnumero2.7 and \textnumero2.8. The family \textnumero2.7 are Fano threefolds as blow-up of a smooth quadric hypersurface $Q\subset \bP^4$ at a complete intersection of two divisors in $|\cO_Q(2)|$. It is known from \cite[Section 4.5]{ACC+} that a general member of the family \textnumero2.7 is K-stable. It is not known whether every smooth member of the family \textnumero2.7 is K-stable.
\end{rem}

\subsection*{Notation.} Throughout the paper, we work over $\bC$. We follow the definitions and notation from \cite{Xu21, ACC+}.

\subsection*{Acknowledgements.} I would like to thank Hamid Abban, Ivan Cheltsov, Kento Fujita, Anne-Sophie Kaloghiros, Andrea Petracci, Chenyang Xu, and Ziquan Zhuang for helpful discussions and comments. The author is partially supported by NSF Grant DMS-2148266 (formerly DMS-2001317).

\section{\textnumero2.8($\mathrm{a}$)}\label{sec:2.8a}

Recall that $\pi: Y = \Bl_p \bP^3 \to \bP^3$ is the blow-up of $\bP^3$ at a point $p$ with exceptional divisor $E\subset Y$. Let $\oS\subset Y$ be a smooth anti-canonical surface. Let $\sigma: X\to Y$ be the double cover branched along $\oS$. Every smooth Fano threefold $X$ of the family \textnumero2.8 arises this way. Our goal is to show that $(Y, c\oS)$ is K-semistable for some $c\in (0, \frac{1}{2})$ which would imply the K-stability of $X$. 

Throughout this section, we assume that $X$ belongs to the subfamily \textnumero2.8(a). In this case, the anti-canonical surface $\oS$ of $Y$ satisfies that $\oS\cap E$ is a smooth conic. Denote by $S:=\pi_* \oS$ a quartic surface in $\bP^3$. Then the assumption that $\oS \cap E$ being smooth is equivalent to saying that $S$ has an ordinary double point (equivalently, an $A_1$-singularity) at $p$. Choose a projective coordinate $[x,y,z,w]$ of $\bP^3$ such that $p=[0,0,0,1]$. Then the equation of $S$ is
\[
S=(f_2(x,y,z)w^2 + f_3(x,y,z)w + f_4(x,y,z)=0),
\]
where $f_i$ is a degree $i$ homogeneous polynomial in $(x,y,z)$, and $f_2$ is a quadratic form of full rank. Let $C_0:=(f_2(x,y,z)=0)\subset\bP^2$ be a smooth conic curve. Consider the $1$-PS $\lambda: \bG_m \to \PGL(4)$ given by $\lambda(t)\cdot [x,y,z,w]=[x,y,z,tw]$. Then it is clear that $\lim_{t\to 0} \lambda(t)_* S = S_0$ where 
\[
S_0= (f_2(x,y,z) w^2 = 0).
\]
Let $\oS_0:= \pi_*^{-1} S_0$. Then  $\pi^*\lambda$ induces a special degeneration $(Y,c\oS)\rightsquigarrow (Y, c\oS_0)$ for $c\in (0,1)$.

\begin{prop}\label{prop:2.8a}
The log Fano pair $(Y,\frac{3}{17}\oS_0)$ is K-semistable. 
\end{prop}

Using $\bT$-log Fano pairs of complexity $1$ (see e.g. Theorem \ref{thm:complexity1}), one can show that this pair is indeed K-polystable, although K-semistability is enough in proving Theorem \ref{thm:main}.

\begin{proof}
We follow the computation of stability thresholds for projective bundles from \cite{ZZ22} (see Lemma \ref{lem:ZZ}). Let $c\in (0,1)$ be a rational number. Let $\phi:Y\to \bP^2$ be the $\bP^1$-bundle induced by central projection from $p$. Denote by $(V,\Delta):=(\bP^2, cC_0)$ and $\Delta_Y:=\phi^*\Delta$. Then $Y\cong \bP_{V}(L^{-1}\oplus\cO_{V})$ where $L=\cO_{\bP^2}(1)$.  Let $r=3-2c$ so that $L\sim_{\bQ}-r^{-1}(K_V+\Delta)$.
Let  $V_\infty$ be the section of $Y$ at infinity. Thus we have
$c\oS_0 = \Delta_Y + 2c V_\infty$. By Lemma \ref{lem:ZZ} with $a=0$ and $b=2c$, we have
\begin{equation}\label{eq:delta-1}
\delta(Y, c\oS_0) = \min\left\{\frac{(3-2c)\delta(V,\Delta)}{\frac{3}{4}\frac{B^{4}-A^{4}}{B^{3}-A^{3}}}, \frac{1}{\frac{3}{4}\frac{B^{4}-A^{4}}{B^{3}-A^{3}}-A}, \frac{1-2c}{B-\frac{3}{4}\frac{B^{4}-A^{4}}{B^{3}-A^{3}}}\right\},
\end{equation}
where $A=r-(1-a) = 2-2c$ and $B=r+(1-b) = 4-4c$. Thus we have
\begin{equation}\label{eq:delta-2}
\frac{3}{4}\frac{B^4-A^4}{B^3-A^3}= \frac{15}{7}A = \frac{15}{7}(2-2c).
\end{equation}
Combining \eqref{eq:delta-1} and \eqref{eq:delta-2} yields
\begin{equation}\label{eq:delta-3}
\delta(Y, c\oS_0) = \min\left\{\frac{28(3-2c)}{45(2-2c)} \delta(V,\Delta), \frac{28}{17(2-2c)}, \frac{28(1-2c)}{11(2-2c)}\right\}.
\end{equation}
Suppose we take $c=\frac{3}{17}$, then it is easy to see that 
\[
\frac{28(3-2c)}{45(2-2c)} = \frac{28}{17(2-2c)}= \frac{28(1-2c)}{11(2-2c)} =1.
\]
Since $c<\frac{3}{4}$, by \cite[Theorem 1.5]{LS14} (see also \cite{Fuj20}) the log Fano pair $(V,\Delta)=(\bP^2, cC_0)$ is K-polystable, which implies  $\delta(V,\Delta)=1$. Thus \eqref{eq:delta-3} becomes $\delta(Y, \frac{3}{17}\oS_0)=1$ which implies K-semistability of $(Y, \frac{3}{17}\oS_0)$ by \cite{FO16, BJ17}.
\end{proof}

\begin{lem}[cf. {\cite[Theorem 1.3]{ZZ22}}]\label{lem:ZZ}
Let $(V,\Delta)$ be a log Fano pair of dimension $n$. Let $L$ be an ample line bundle on $V$ such that $L\sim_{\bQ} -r^{-1} (K_V+\Delta)$ for some $r\in\bQ_{>0}$. Let $\phi:Y=\bP_V(L^{-1}\oplus \cO_V)\to V$ be a $\bP^1$-bundle. Denote by $\Delta_Y:=\phi^*\Delta$. Let $V_0$ and $V_\infty$ be sections of $Y$ at zero and infinity respectively, so that $\cO_Y(V_0)|_{V_0}\cong L^{-1}$ and $\cO_Y(V_\infty)|_{V_\infty}\cong L$. Let $a,b$ be rational numbers such that $0\leq a<1$, $0\leq b <1$ if $r>1$, and $1-r<a<1$, $0\leq b <1$ if $0<r\leq 1$. Then
\[
\delta(Y,\Delta_Y+aV_0+bV_\infty)= \min\left\{\frac{r \delta(V,\Delta)}{\frac{n+1}{n+2}\frac{B^{n+2} - A^{n+2}}{B^{n+1} - A^{n+1}}}, \frac{1-a}{\frac{n+1}{n+2}\frac{B^{n+2} - A^{n+2}}{B^{n+1} - A^{n+1}} - A}, \frac{1-b}{B - \frac{n+1}{n+2}\frac{B^{n+2} - A^{n+2}}{B^{n+1} - A^{n+1}}}\right\}, 
\]
where $A=r-(1-a)$ and $B=r+(1-b)$.
\end{lem}

\begin{proof}
The proof is the same as \cite[Section 3]{ZZ22} after replacing  $V$ and $Y$ therein by $(V, \Delta)$ and $(Y, \Delta_Y)$ respectively.
\end{proof}

\begin{cor}\label{cor:2.8a}
Assume the smooth anti-canonical surface $\oS$ of $Y$ satisfies that $\oS\cap E$ is a smooth conic. Then $(Y, \frac{3}{17}\oS)$ is K-semistable.
\end{cor}

\begin{proof}
Since $\pi^*\lambda$ induces a special degeneration $(Y, \frac{3}{17}\oS)\rightsquigarrow (Y, \frac{3}{17}\oS_0)$, the statement follows from Proposition \ref{prop:2.8a} and the openness of K-semistability \cite{BLX19, Xu19}.
\end{proof}

\section{\textnumero2.8($\mathrm{b}$)}
We adapt the notation from the first paragraph of Section \ref{sec:2.8a}. 

Throughout this section, we assume that $X$ belongs to subfamily \textnumero2.8(b).
In this case, the anti-canonical surface $\oS$ of $Y$ satisfies that $\oS\cap E$ is a union of two transversal lines. Denote by $S:=\pi_* \oS$ a quartic surface in $\bP^3$. Choose a projective coordinate $[x,y,z,w]$ of $\bP^3$ such that $p=[0,0,0,1]$ and the equation of $S$ is 
\[
S=(xyw^2 + f_3(x,y,z)w + f_4(x,y,z)=0),
\]
where $f_i$ is a degree $i$ homogeneous polynomial in $(x,y,z)$.

\textbf{Claim.} The polynomial $f_3$ has a non-zero $z^3$-term. This is equivalent to saying that $S$ has an $A_2$-singularity at $p$.

We choose an affine coordindate $[x,y,z,1]$ on $\bP^3$ and an affine coordinate $(x_0,x_1,x_2)$ on $Y$ such that $(x,y,z)=(x_0x_2, x_1x_2, x_2)$. Then the equation of $\oS$ in the coordinate $(x_0, x_1, x_2)$ becomes 
\[
x_0 x_1 + f_3(x_0,x_1,1) x_2 + f_4(x_0, x_1, 1) x_2^2 = 0.
\]
Since $\oS$ is smooth at the origin of the coordinate $(x_0, x_1, x_2)$, we conclude that $f_3(0,0,1)\neq 0$ which implies that $f_3$ has a non-zero $z^3$-term. The claim is proved.

Next, after rescaling of $z$ we may assume that the $z^3$-term has coefficient $1$ in $f_3$. Let $\lambda':\bG_m\to \PGL(4)$ be a $1$-PS given by $\lambda'(t)\cdot [x,y,z,w] = [x, y, tz, t^3w]$. Then it is clear that $\lim_{t\to 0} \lambda'(t)_* S = S_0'$ where 
\[
S_0' = (xyw^2 + z^3 w=0).
\]
Let $\oS_0':= \pi_*^{-1} S_0'$. Then  $\pi^*\lambda'$ induces a special degeneration $(Y,c\oS)\rightsquigarrow (Y, c\oS_0')$ for $c\in (0,1)$.

\begin{prop}\label{prop:2.8b}
The log Fano pair $(Y,\frac{2}{9}\oS_0')$ is K-polystable. 
\end{prop}

\begin{proof}
Let $\mu$ be a $1$-PS in $\PGL(4)$ given by $\mu(t)\cdot[x,y,z,w]=[tx, t^{-1}y, z, w]$. Then clearly $\lambda'$ and $\mu$ generates a $\bT=\bG_m^2$-action on $(\bP^3, S_0')$ which lifts through $\pi^*$ to a $\bT$-action on $(Y, \oS_0')$. Since the $\bT$-action on $Y$ is of complexity $1$, by Theorem \ref{thm:complexity1} we only need to check $\beta_{(Y,\frac{2}{9}\oS_0')}(F)$ and $\Fut_{(Y,\frac{2}{9}\oS_0')}$. 

We first show that $\beta_{(Y,\frac{2}{9}\oS_0')}(F)>0$ for every $\bT$-invariant prime divisor $F$ on $Y$. In fact, all $\bT$-invariant prime divisors on $Y$ are vertical from the following classification. Denote by $\oH_x, \oH_y, \oH_z,\oH_w$ the strict transform of the coordinate hyperplanes of $\bP^3$ to $Y$. A straightforward analysis of the $\bT$-action on $\bP^3$ shows that every $\bT$-invariant prime divisor on $Y$ belongs to one of the following classes:
\begin{enumerate}[label=(\roman*)]
    \item $E$;
    \item $\oH_w$;
    \item $\oH_x$, $\oH_y$, $\oH_z$;
    \item $\oT_s:=\pi_*^{-1} T_s$ where $T_s:=(xyw + sz^3=0)\subset\bP^3$ for $s\neq 0$.
\end{enumerate}

Next, we split into four cases according to the above classes. We will frequently use the equality 
\[
S_{(Y, cD)}(F)= (1-c) S_Y(F)
\]
for an anti-canonical surface $D\subset Y$ and $c\in (0,1)$ which follows from $-K_Y-cD\sim_{\bQ} (1-c)(-K_Y)$ and \cite[Lemma 3.7(i)]{BJ17}.

(i) It is clear that $A_{(Y, \frac{2}{9}\oS_0')}(E)=A_Y(E) = 1$. Next we compute $S_Y(E)$. We know that $-K_Y \sim 4\oH_w - 2E$ where $\oH_w\sim \pi^*\cO_{\bP^3} (1)$. The pseudoeffective cone of $Y$ is generated by $\oH_x\sim \oH_w - E$ and $E$. The nef cone of $Y$ is generated by $\oH_w$ and $\oH_x$. Thus we have $-K_Y-tE\sim 4\oH_w -(2+t)E$ is nef if $0\leq t\leq 2$, and not big if $t\geq 2$. Thus
\begin{align*}
S_Y(E) &  = \frac{1}{(-K_Y)^3} \int_0^\infty \vol(-K_Y-tE) dt \\
& = \frac{1}{56} \int_0^2 (4\oH_w - (2+t) E)^3 dt\\
& = \frac{1}{56} \int_0^2 (4^3 - (2+t)^3) dt = \frac{17}{14}.
\end{align*}
As a consequence,
\[
\beta_{(Y,\frac{2}{9}\oS_0')}(E) = 1 - \frac{7}{9}\cdot \frac{17}{14}=\frac{1}{18} >0.
\]

(ii) We have  $A_{(Y, \frac{2}{9}\oS_0')}(\oH_w)=A_Y(\oH_w)-\frac{2}{9}\ord_{\oH_w}(\oS_0') = \frac{7}{9}$. Next we compute $S_Y(\oH_w)$. It is clear that $-K_Y-t\oH_w = (4-t)\oH_w - 2E $ is nef if $0\leq t\leq 2$, and not big if $t\geq 2$. Thus %\footnote{revision: changed $E$ from $\vol(-K_Y-tE)$ to the corresponding divisors in cases (ii),(iii) and (iv).}
\begin{align*}
S_Y(\oH_w) &  = \frac{1}{(-K_Y)^3} \int_0^\infty \vol(-K_Y-t\oH_w) dt \\
& = \frac{1}{56} \int_0^2 ((4-t)\oH_w - 2 E)^3 dt\\
& = \frac{1}{56} \int_0^2 ((4-t)^3 - 2^3) dt = \frac{11}{14}.
\end{align*}
As a consequence, 
\[
\beta_{(Y,\frac{2}{9}\oS_0')}(\oH_w) = \frac{7}{9} - \frac{7}{9}\cdot \frac{11}{14}=\frac{1}{6} >0.
\]

(iii) Since $\oH_x\sim \oH_y\sim\oH_z$, their $S$-invariants are the same. As neither of these three divisors is  contained in $\Supp(\oS_0')$, their log discrepancies with respect to $(Y, \frac{2}{9}\oS_0')$ are the same which is $1$. It suffices to show that $\beta(\oH_x)>0$ as these three divisors have the same $\beta$-invariant. Next we compute $S_Y(\oH_x)$. The divisor $-K_Y - t\oH_x = (4-t)\oH_w - (2-t) E$ is nef if $0\leq t\leq 2$ and not big if $t\geq 4$. When $2< t< 4$, it admits a Zariski decomposition
\[
-K_Y - t \oH_x = (4-t)\oH_w + (t-2) E,
\]
which implies that $\vol(-K_Y - t\oH_x) = ((4-t)\oH_w)^3$ when $2\leq t\leq 4$. Thus we have
\begin{align*}
S_Y(\oH_x) &  = \frac{1}{(-K_Y)^3} \int_0^\infty \vol(-K_Y-t\oH_x) dt \\
& = \frac{1}{56}\left( \int_0^2 ((4-t)\oH_w - (2-t) E)^3 dt +\int_2^4 ((4-t)\oH_w)^3 dt\right)\\
& = \frac{1}{56} \left( \int_0^2 ((4-t)^3 - (2-t)^3) dt +\int_2^4 (4-t)^3 dt\right) = \frac{15}{14}.
\end{align*}
As a consequence,
\[
\beta_{(Y,\frac{2}{9}\oS_0')}(\oH_x) = 1 - \frac{7}{9}\cdot \frac{15}{14}=\frac{1}{6} >0.
\]

(iv) Since $\oT_s \sim 3\oH_w - 2E$ for every $s\neq 0$, their $S$-invariant are the same. We also have $\ord_{\oT_s}(\oS_0')\leq 1$, which implies $A_{(Y, \frac{2}{9}\oS_0')}(\oT_s)\geq \frac{7}{9}$. Next we compute $S_Y(\oT_s)$. The divisor $-K_Y-t\oT_s=(4-3t)\oH_w - (2-2t)E$ is nef if $0\leq t\leq 1$ and not big if $t\geq \frac{3}{4}$. When $1<t<\frac{3}{4}$, it admits a Zariski decomposition
\[
-K_Y - t \oT_s = (4-3t)\oH_w + (2t-2) E,
\]
which implies that $\vol(-K_Y - t\oH_x) = ((4-3t)\oH_w)^3$ when $1\leq t\leq\frac{4}{3}$. Thus we have
\begin{align*}
S_Y(\oT_s) &  = \frac{1}{(-K_Y)^3} \int_0^\infty \vol(-K_Y-t\oT_s) dt \\
& = \frac{1}{56}\left( \int_0^1 ((4-3t)\oH_w - (2-2t) E)^3 dt +\int_1^{\frac{4}{3}} ((4-3t)\oH_w)^3 dt\right)\\
& = \frac{1}{56} \left( \int_0^1 ((4-3t)^3 - (2-2t)^3) dt +\int_1^{\frac{4}{3}} (4-3t)^3 dt\right) = \frac{29}{84}.
\end{align*}
As a consequence,
\[
\beta_{(Y,\frac{2}{9}\oS_0')}(\oT_s) \geq \frac{7}{9} - \frac{7}{9}\cdot \frac{29}{84}=\frac{55}{108} >0.
\]

So far we verified that $\beta_{(Y,\frac{2}{9}\oS_0')}(F)>0$ for every $\bT$-invariant prime divisor $F$ on $Y$. It remains to show that $\Fut_{(Y,\frac{2}{9}\oS_0')}=0$ on the cocharacter lattice $N=\Hom(\bG_m,\bT)$ of $\bT$. Let $\lambda_0$ and $\lambda_1$ be two $1$-PS in $\PGL(4)$ given by 
\[
\lambda_0(t)\cdot [x,y,z,w]=[t^3x, y, tz, w] \quad\textrm{and}\quad \lambda_1(t)\cdot[x,y,z,w]=[x, t^3y, tz, w].
\]
Then it is not hard to see that $\lambda_0$ and $\lambda_1$ form a basis of $N_{\bQ}:=N\otimes_{\bZ}\bQ$. Meanwhile, the involution $\tau: \bP^3\to \bP^3$ defined by $[x,y,z,w]\mapsto [y,x,z,w]$ induces an involution $\pi^*\tau \in \Aut(Y, \oS_0')$ such that $\tau\lambda_0\tau^{-1} = \lambda_1$. Thus it suffices to show that $\Fut_{(Y,\frac{2}{9}\oS_0')}(\pi^*\lambda_0)=0$ as this implies $\Fut_{(Y,\frac{2}{9}\oS_0')}(\pi^*\lambda_1)=0$ and hence the vanishing of $\Fut_{(Y,\frac{2}{9}\oS_0')}$ on the entire $N_{\bQ}$.

Let $v$ be the monomial divisorial valuation on $\bP^3$ centered at $(x=z=0)$ such that $v(x)=3$ and $v(z)=1$. Then it is clear that $\lambda_0$ is the $1$-PS induced by $v$. As abuse of notation we also denote by $v$ the lifting valuation $\pi^*v$ on $Y$. According to \cite[Theorem 5.1]{Fuj16}, we have $\Fut_{(Y,\frac{2}{9}\oS_0')}(\lambda_0)=\beta_{(Y,\frac{2}{9}\oS_0')}(v)$. Thus it suffices to show $\beta(v)=0$. Since $\pi: Y\to \bP^3$ is isomorphic at the generic point of the center of $v$, we know that 
\begin{equation}\label{eq:Av}
A_{(Y,\frac{2}{9}\oS_0')}(v) = A_{(\bP^3, \frac{2}{9}S_0')}(v) = A_{\bP^3}(v) - \frac{2}{9}\cdot v(S_0')= 4-\frac{2}{9}\cdot 3 = \frac{10}{3}.    
\end{equation}
Next, we compute $S_Y(v)$. Let $F_0$ be the exceptional divisor of the $(3,0,1)$-weighted blow up in the affine $(x,y,z)$ with $w=1$. Then it is clear that $v=\ord_{F_0}$. Thus we have $\vol(-K_Y-tv) = \vol(\cO_{\bP^3}(4)-2E - tF_0)$. As both $E$ and $F_0$ are toric divisors over $\bP^3$, we have 
\[
\vol(\cO_{\bP^3}(4)-2E - tF_0) = 3!\cdot\vol(P_t), 
\]
where 
\[
P_t := \{(u_0,u_1,u_2)\in \bR_{\geq 0}^3\mid 2\leq u_0+u_1+u_2\leq 4\textrm{ and }3u_0+u_2\geq t\}.
\]
Let 
\[
Q_t := \{(u_0,u_1,u_2)\in \bR_{\geq 0}^3\mid u_0+u_1+u_2\leq 1\textrm{ and }3u_0+u_2\geq t\}.
\]
Then it is clear that $\vol(P_t) = 4^3 \vol(Q_{\frac{t}{4}}) - 2^3 \vol(Q_{\frac{t}{2}})$. Using convex geometry it is not hard to show that 
\[
\vol(Q_t) = \begin{cases}
\frac{1}{6} -\frac{1}{6}t^2 + \frac{2}{27}t^3 & \textrm{if }0\leq t\leq 1;\\
\frac{1}{108}(3-t)^3 & \textrm{if } 1\leq t\leq 3;\\
0 & \textrm{if }t\geq 3.
\end{cases}
\]
Computation shows that $\int_0^3 \vol(Q_t) dt= \frac{1}{6}$.
Thus we have 
\begin{align*}
    \int_0^\infty \vol(-K_Y - tv) dt & = \int_0^\infty 6\vol(P_t) dt \\
    & = \int_0^{12}6\cdot 4^3 \vol(Q_{\frac{t}{4}}) dt - \int_0^6 6\cdot 2^3 \vol(Q_{\frac{t}{2}}) dt \\
    & = 6\cdot(4^4-2^4) \int_0^3 \vol(Q_t)dt = 240.
\end{align*}
As a result, we have
\begin{equation}\label{eq:Sv}
    S_{(Y,\frac{2}{9}\oS_0')}(v) = \frac{7}{9}S_Y(v) = \frac{7}{9}\cdot \frac{1}{56}\int _0^\infty \vol(-K_Y - tv) dt = \frac{7}{9}\cdot \frac{1}{56}\cdot 240 = \frac{10}{3}.
\end{equation}
Combining \eqref{eq:Av} and \eqref{eq:Sv}, we get $\beta_{(Y,\frac{2}{9}\oS_0')}(v)=\frac{10}{3}-\frac{10}{3}=0$. Thus the proof is finished.
\end{proof}

The following theorem is a logarithmic version of a result in \cite{ACC+} which originated from \cite{IS17}. There is little change to the proof so we omit it.
\begin{thm}[cf. {\cite[Theorem 1.3.9]{ACC+}}]\label{thm:complexity1} Let $(X,\Delta)$ be a log Fano pair with an algebraic torus $\bT$-action of complexity $1$. Then $(X,\Delta)$ is K-polystable if and only if all of the following conditions hold.
\begin{enumerate}
    \item $\beta_{(X,\Delta)}(F)>0$ for every vertical $\bT$-invariant prime divisor $F$ on $X$;
    \item $\beta_{(X,\Delta)}(F)=0$ for every horizontal $\bT$-invariant prime divisor $F$ on $X$;
    \item $\Fut_{(X,\Delta)}=0$ on the cocharacter lattice of $\bT$.
\end{enumerate}

\end{thm}

\begin{cor}\label{cor:2.8b}
Assume the smooth anti-canonical surface $\oS$ of $Y$ satisfies that $\oS\cap E$ is a union of two transversal lines. Then $(Y, \frac{2}{9}\oS)$ is K-semistable.
\end{cor}

\begin{proof}
Since $\pi^*\lambda'$ induces a special degeneration $(Y, \frac{2}{9}\oS)\rightsquigarrow (Y, \frac{2}{9}\oS_0')$, the statement follows from Proposition \ref{prop:2.8b} and the openness of K-semistability \cite{BLX19, Xu19}.
\end{proof}

\begin{proof}[Proof of Theorem \ref{thm:main}]
Let $\sigma:X\to Y$ be the double cover branched along a smooth anti-canonical surface $\oS\subset Y$. By \cite[Theorem 1.3]{Der16}, \cite[Theorem 1.2]{LZ22}, and \cite[Corollary 4.13]{Zhu20}, it suffices to show K-stability of $(Y,\frac{1}{2}\oS)$ as $\Aut(X)$ is finite according to \cite[Lemma 12.4]{CPS19}. By \cite[Theorem 2.10]{ADL21} (see also \cite[Corollary 1]{JMR16}), we know that $(Y, (1-\epsilon)\oS)$ is K-stable for $0<\epsilon\ll 1$. Combining Corollaries \ref{cor:2.8a} and \ref{cor:2.8b}, we know that $(Y, c\oS)$ is K-semistable for some $c\in (0, \frac{1}{2})$ (more precisely, $c=\frac{3}{17}$ in family \textnumero2.8(a) or $c=\frac{2}{9}$ in family \textnumero2.8(b)). Thus the interpolation of K-stability \cite[Proposition 2.13]{ADL19} implies that $(Y, \frac{1}{2}\oS)$ is K-stable. The existence of K\"ahler-Einstein metrics follows from \cite{CDS15, Tia15}. Thus the proof is finished.
\end{proof}

\begin{rem}
Our arguments can give some K-polystable and K-semistable singular members in the family \textnumero2.8 as well. If a quartic surface $S\subset\bP^3$ has an $A_1$ or $A_2$-singularity at $p$ and is canonical (resp. semi-log-canonical) elsewhere, then similar arguments show that $(Y, \frac{1}{2}\oS)$ is K-stable (resp. K-semistable) which implies that the double cover $X\to (Y,\frac{1}{2}\oS)$ %\footnote{deleted an extra ``)''} 
is K-polystable (resp. K-semistable).

It is an interesting problem to describe the boundary of the K-moduli  compactification of all smooth Fano threefolds in the family \textnumero2.8. For comparison, see \cite[Theorem 1.4]{ADL21} where a complete description of the K-moduli compactification of quartic double solids is given.
\end{rem}

\bibliographystyle{alpha}
\bibliography{ref}

\end{document}